\documentclass[11pt]{article}

\usepackage{setspace}
\usepackage[a4paper,margin=1in]{geometry}
\usepackage[T1]{fontenc}
\usepackage{manish}
\usepackage[utf8]{inputenc}
\usepackage{libertinus}
\usepackage{amsmath,amssymb,amsthm,mathtools}
\usepackage{bbm}
\usepackage{xcolor}

\usepackage{enumitem}
\usepackage[numbers,sort&compress]{natbib}
\usepackage{hyperref}
\usepackage[nameinlink,capitalize,noabbrev]{cleveref}
\usepackage{comment}
\usepackage[normalem]{ulem}
\usepackage{setspace}

\hypersetup{
    colorlinks=true,
    linkcolor=blue,
    citecolor=blue,
    urlcolor=blue
}

\theoremstyle{plain}

\newif\iflongversion
\longversiontrue 

\newcommand{\longversion}[1]{\iflongversion#1\fi}
\newcommand{\shortversion}[1]{\iflongversion\else#1\fi}

\newcommand\blfootnote[1]{%
  \begingroup
  \renewcommand\thefootnote{}\footnote{#1}%
  \addtocounter{footnote}{-1}%
  \endgroup
}

 \author{Souvik Dhara\thanks{H. Milton Stewart School of Industrial \& Systems Engineering, Georgia Institute of Technology. {\tt sdhara@gatech.edu}}, \ Manish Pandey\thanks{Department of Mathematics,  Aarhus University. {\tt m.pandey@math.au.dk}}, \ and \ Leonard J. Schulman\thanks{California Institute of Technology, Pasadena CA 91125, USA. 
                 {\tt schulman@caltech.edu.} Supported in part by NSF grant CCF-2321079 and DARPA Agreement HR0011-25-3-0213.}}
\date{}

\title{Cascade Surviving Cores via Local Limits}

\begin{document}

\maketitle

\begin{abstract}
We consider a general class of cascading processes on marked networks, in which vertices are iteratively deleted according to a local vertex property. Under marked local--weak convergence, we establish a necessary and sufficient condition for the proportion of vertices in the final stable set to converge to the survival probability in the limiting network. Under this condition, we also obtain convergence of the proportions of vertices in subsets specified by local properties.

\end{abstract}
\blfootnote{\textbf{Acknowledgments.} 
SD and MP would like to acknowledge the summer internship program at the Edwardson School of Industrial Engineering, Purdue University, where a part of the work was done. 
}\blfootnote{\textbf{AI disclosure:} AI tools were used solely for minor  polishing and did not contribute to the technical or conceptual content of this work.}
Consider a graph $G=(V,E)$ in which each vertex is assigned a mark from some Polish space. Each vertex is endowed with a \emph{cascading property} that depends on its own mark together with the marks of vertices observed in its local neighborhood. A cascading process is an iterative deletion procedure in which all vertices that do not satisfy the cascading property are removed. Such deletions may in turn cause additional vertices to violate the cascading property, leading to further deletions in subsequent steps. This induces an iterative vertex-deletion process that continues until stability is reached, namely, until every vertex in the residual set satisfies the cascading property. Our primary interest lies in the convergence properties of the functionals of the final residual set, and of its largest connected components.

Numerous related problems fall within this general framework. For example, the \emph{$k$-core problem}, which involves iteratively deleting vertices of degree at most $k-1$, has been studied extensively in random graphs~\cite{pittel1996sudden,riordan2008k,fernholzgiant,janson2008asymptotic,Luc91}. 
Related cascading procedures, including unit clause propagation for Boolean satisfiability, are studied within the Warning Propagation framework~\cite{cooley2021warning}.
Our treatment of the largest surviving component is motivated by, and analogous to, the classical \emph{giant component problem}~\cite{Draief;Massoulie:2009,Jackson:2010,Hofstad:2017,Hofs24}.
More general cascading processes have appeared in recent works in economic theory~\cite{EGL22,elliott2022networks}, where the network and vertex marks model interdependencies among units across different sectors of the economy, and the cascading process represents sequential failures triggered by disruptions in the system.

The technical framework is motivated by the recent work of van der Hofstad~\cite{van2021giant}, which established necessary and sufficient conditions under which Benjamini--Schramm local--weak convergence~\cite{BenSch01,AldSte04} implies concentration and uniqueness of the giant component. The main contribution of this work is to establish an analogous result for a broad class of cascading processes resulting from local vertex properties on sparse random networks.

\section{Preliminaries}

\paragraphown{Marked local--weak convergence.}
Fix a Polish mark space $(\mathscr{M},d_{\mathscr{M}})$. A \emph{marked, rooted, locally finite graph} is a triple
\[
(G,o,\sigma)\equiv \bigl(G=(V,E),\,o\in V,\,\sigma:V\to\mathscr{M}\bigr),
\]
where $G$ is connected and each of its vertices has finite degree. We identify two such triples if there  exists an isomorphism that preserves the root and marks, and write $\Gm$ for the corresponding space of equivalence classes. For $r\in\mathbb{N}$, let
\[
\pi_r(G,o,\sigma):=(G,o,\sigma)^{\le r}
:=\bigl(G[B_r(G,o)],\,o,\,\sigma|_{B_r(G,o)}\bigr)
\]
denote the rooted marked radius-$r$ neighborhood, where $B_r(G,o)$ denotes the set of vertices within graph distance $r$ from $o$ in $G$, and $G[U]$ denotes the induced subgraph on $U$. 
The marked neighborhood topology on $\Gm$ compares finite rooted balls through root-preserving graph isomorphisms under which corresponding marks are close in $d_{\mathscr M}$; see \cite{Hofs24} for a metric formulation and the standard properties, including separability and completeness.
When $G$ is disconnected, the notation $(G,o,\sigma)$ refers to the connected component containing $o$, rooted at $o$ and equipped with the restricted marks. For a finite marked graph $(G_n,\sigma_n)$, its empirical rooted neighborhood measure is
\[
\widehat{\mu}_{G_n}:=\frac1{|V(G_n)|}\sum_{v\in V(G_n)}\delta_{(G_n,v,\sigma_n)}\in\mathcal{P}(\Gm),
\] 
where $\mathcal{P}(X)$ denotes the set of Borel probability measures on $X$, and $\delta_{\cdot}$ denotes the Dirac mass. 
When $G_n$ is a random graph, $\widehat{\mu}_{G_n}$ is a $\mathcal{P}(\Gm)$-valued random variable. We say that a sequence of marked random graphs $(G_n,\sigma_n)_{n\geq 1}$ converges to $\mu\in\mathcal{P}(\Gm)$ in
\emph{marked local--weak convergence in probability} if, for every bounded
continuous function $f:\Gm\to\mathbb{R}$,
\[
\left\langle \widehat{\mu}_{G_n},f\right\rangle
\xrightarrow{\scriptscriptstyle\prob}
\left\langle \mu,f\right\rangle.
\]
This is equivalent to having $\hat{\mu}_{G_n} (A) \xrightarrow{\scriptscriptstyle\prob} \mu (A)$ for all $\mu$-continuity sets $A$.

\paragraphown{Cascade and Surviving core.}
A \emph{vertex property} is a measurable map $P$ assigning to each marked graph $(G,m)$ and vertex $v\in V(G)$ a value $P(v;G,m)\in\{0,1\}$. We say that \(v\) \emph{satisfies \(P\)} in \((G,m)\) if \(P(v;G,m)=1\), and that \(v\) \emph{fails \(P\)} in \((G,m)\) if \(P(v;G,m)=0\). We often write \(v\) satisfies or does not satisfy $P$ in $G$ when the marks are clear from context.

Fix such a property $P$ and a marked graph $(G,m)$. Set $V_0:=V(G)$ and $(G_0,m_0):=(G,m)$, and define,
for $t\ge0$,
\[
V_{t+1}:=\{v\in V(G_t):P(v;G_t,m_t)=1\},\qquad
G_{t+1}:=G_t[V_{t+1}],\qquad
m_{t+1}:=m_t|_{V_{t+1}}.
\]
The \emph{cascade surviving core}, or \emph{$P$-core}, is
\[
G_{\infty} = G_{\infty}(G,m;P):= G\bigg[\bigcap_{t\ge0}V_t\bigg]. 
\]
Thus, at each step $t$, all vertices failing to satisfy $P$ are deleted synchronously. The synchronous deletion gives a canonical final core.

\paragraphown{Local Cascade Survival.} For an integer $r\ge0$, a property $P$ is called \emph{$r$-local} if, for every vertex~$v$, $P(v;G,m)$ depends only on the rooted marked $r$-neighborhood of $v$. Throughout the main results, $P$ is assumed to be $r$-local for some fixed $r\ge1$.

For a marked graph $(G,m)$, let $(V_t(G,m;P))_{t\ge0}$ be the synchronous
cascade process defined above. 
We say that $v$ \emph{survives through $R$ deletion rounds} if \(v\in V_R(G,m;P).\)

\begin{lemma}[Locality of finite-round survival]
\label{lem:finite-round-locality}
If $P$ is $r$-local, then whether $v$ belongs to $V_R(G,m;P)$ or not can be determined by the rooted marked radius-$rR$ neighborhood of $v$.
\end{lemma}
\begin{proof}
We argue by induction on $R$. The assertion is trivial for $R=0$. Suppose it
holds for $R$. To decide whether $v$ survives round $R+1$, it is enough to know
which vertices in the radius-$r$ neighborhood of $v$ survive the first $R$
rounds. For each such vertex $u$, the induction hypothesis requires only
$B_{rR}(G,u)$, and $B_{rR}(G,u)\subseteq B_{r(R+1)}(G,v).$
Thus the radius-$r(R+1)$ neighborhood of $v$ determines whether
$v\in V_{R+1}(G,m;P)$.
\end{proof}

\paragraphown{Sequential Core and Maximum Stable Subgraph.}
We now show that, for a class of properties called monotonically non-decreasing properties, the surviving core is the same whether vertices are deleted synchronously or one at a time according to a given rule. Although this framework is not needed for the main results, it is useful for describing cascades that occur sequentially, for example, when the vertices get deleted after exponentially distributed times.

For marked graphs $(G,m)$ and $(H,\widetilde m)$, write $(G,m)\preceq (H,\widetilde m)$ if $G$ is an induced subgraph of $H$ and $\widetilde m|_{V(G)}=m$. We call $P$ \emph{monotone non-decreasing} (MND) if
\[
(G,m)\preceq (H,\widetilde m)\ \text{ and }\ P(v;G,m)=1
\quad\Longrightarrow\quad
P(v;H,\widetilde m)=1.
\]
For example, the property $P(v;G,m)=\mathbbm{1}_{\{\deg_G(v)\ge k\}}$, which gives rise to the $k$-core, is MND. The lemma below connects all the above concepts to the cascade surviving core $G_{\infty}$:

An induced subgraph \(H \subseteq G\) is called \emph{a stable subgraph} if
\(
P(v;H,m|_{V(H)})=1, \text{ for every }v\in V(H).
\)
We call a stable induced subgraph $H^\star$ \emph{a maximum stable subgraph with respect to inclusion}
if every stable induced subgraph of $G$ is contained in $H^\star$. Such a subgraph, if it exists, is unique.

Next, let $(x_k)_{k\ge0}$ be any sequence of vertices in which every vertex
of $G$ occurs infinitely often. 
The \emph{sequential core} $W_\infty$ is then obtained as follows: Set $W_0:=V(G)$ and define
\[
W_{k+1}:=
\begin{cases}
W_k\setminus\{x_k\},&\text{if }x_k\in W_k\text{ and }
P(x_k;G[W_k],m|_{W_k})=0,\\
W_k,&\text{otherwise}.
\end{cases}
\]
Let $W_{\infty}:= \cap_{k\geq 0} W_k$. Below we connect these concepts to the cascade surviving core $G_{\infty}$:
\begin{lemma}[Maximum stable subgraph and order independence]
\label{lem:maximal-core}
Let $(G,m)$ be such that $G$ is locally finite with a finite or countable vertex set, and let $P$ be
MND and $r$-local for some $r\ge1$. 
Then $G_\infty=G_\infty(G,m;P)$ is the maximum stable subgraph with respect to {inclusion}. Moreover, $V(G_\infty) = W_\infty$, irrespective of the choice of $(x_k)_{k\geq 0}$.
\end{lemma}

\begin{proof} First, every vertex in $V(G_\infty)$ satisfies $P$ in $G_\infty$. Indeed,
for fixed $v\in V(G_\infty)$, local finiteness ensures that $B_r(G,v)$ is finite. Hence
all vertices of this ball that are deleted by the synchronous process are
deleted by some common finite time $T>0$. The rooted radius-$r$ neighbourhood of
$v$ therefore remains unchanged after $T$, and $B_r(G_{T+1},v)$ agrees with $B_r(G_\infty,v)$. 
If
$v$ failed $P$ in $G_\infty$, $r$-locality would imply that it failed $P$ at
some finite stage, contradicting $v\in V(G_\infty)$.

Now let $H\subseteq G$ be stable.  If $H\subseteq G_t$, then \((H,m|_{V(H)})\preceq(G_t,m_t)\), so MND implies that every vertex of $H$ satisfies $P$ in
$G_t$. Thus $H\subseteq G_{t+1}$, and induction gives $H\subseteq G_\infty$.
This proves the first part.

For the process $(W_k)_{k\ge0}$, we first show by
induction that $V(G_\infty)\subseteq W_k$ for every $k\ge0.$ 
The claim is immediate for \(k=0\). Suppose it holds for some \(k\).
If 
\(x_k\in V(G_\infty)\), then \(x_k\) satisfies \(P\) in \(G_\infty\).
Since $P$ is MND, $x_k$ must satisfy $P$ in $G[W_k]$ too. 
Thus \(x_k\) is not removed, and therefore
\(V(G_\infty)\subseteq W_{k+1}\).

Conversely, fix $v\in W_\infty$, and suppose, for contradiction, that $v$ fails $P$ in
$G[W_\infty]$.
Since $B_r(G,v)$ is finite, there exists a finite $k_0$ such that
all vertices in $B_r(G,v)\setminus W_\infty$ have been removed by
step $k_0$. Hence, for every $k\ge k_0$, the rooted marked
radius-$r$ neighbourhood of $v$ in $G[W_k]$ agrees with that in
$G[W_\infty]$.  Since $P$ is $r$-local, it follows that
\[
P(v;G[W_k],m|_{W_k})=0
\qquad\text{for every } k\ge k_0.
\]
Since every vertex $v$ occurs infinitely often in the sequence
$(x_k)_{k\ge0}$, if $k' \geq k_0$ is the first time when $x_{k'}=v$, the update rule
would remove $v$ at step $k'$. This contradicts $v\in W_\infty$. Hence every vertex of $G[W_\infty]$ must satisfy $P$, so
$G[W_\infty]$ is stable. By the maximality assertion proved above,
$W_\infty\subseteq V(G_\infty)$.
\end{proof}

\begin{remark}
Both the MND assumption and local finiteness are necessary for Lemma~\ref{lem:maximal-core}. For the necessity of MND, consider the path graph $1-2-3-4-5$, 
and let the cascading property be $P(v;G)=\mathbf{1}_{\{\deg_G(v)\le 1\}}$.
Then \(G_\infty\) contains only the vertices \(1\) and \(5\), whereas the maximal induced subgraph satisfying the cascading property contains the vertices \(1,2,4,\) and \(5\). 
For the necessity of local finiteness, consider the \emph{spider with growing arms}, 
described as follows: the vertex set is $V = \{o\} \cup \{v_{i,j} : i \in \mathbb{N},\, 
1 \leq j \leq i\}$ and the edge set is $E = \{\{o, v_{i,1}\} : i \geq 1\} \cup 
\{\{v_{i,j}, v_{i,j+1}\} : i \geq 1,\, 1 \leq j \leq i-1\}$, so that the $i$-th 
arm is a path of length $i$ emanating from $o$. In this graph, $o$ has infinite degree, 
violating local finiteness. Consider the cascade on this graph with property $P(v; G) = 
\mathbf{1}\{\deg_G(v) \geq 2\}$. 
Every $G[V_{t+1}]$ is isomorphic to $G[V_t]$, but nevertheless each $v_{ij}$ is eventually deleted. 
However, at every finite stage $t$, the center $o$ retains a neighbor in all arms of original length $> t$, and therefore $o$ lies in $\cap_{t\geq 0} V_t$.
On the other hand, there is no nonempty subset $H \subseteq V$ in which every vertex has induced degree at least $2$.

\end{remark}

\section{Main Results}

Let $((\mathcal{G}_n,m_n))_{n\ge1}$ be a sequence of finite random marked graphs with $|V(\mathcal{G}_n)|=n$, and let $o_n$ be a uniformly chosen vertex of $\mathcal{G}_n$, conditionally on $(\mathcal{G}_n,m_n)$. For $\mu\in\mathcal{P}(\Gm)$ and $(\mathcal{G},o,m)\sim\mu$, set
\[
\eta_R:=\mu(o\in V_R(\cG,m;P)),
\qquad
\eta:=\lim_{R\to\infty}\eta_R.
\]
By construction, $V_{R+1}(\mathcal{G},m;P)\subseteq V_R(\mathcal{G},m;P)$, 
for every vertex property $P$. Hence $(\eta_R)_{R\ge0}$ is
non-increasing, and the limit defining $\eta$ always exists.
Although $\eta$ depends on $\mu$ and $P$, we suppress this dependence in the notation.
Denote the final surviving vertex set in $(\mathcal{G}_n,m_n)$ by
\(
\mathcal{CS}_{P,n}:=V\bigl(G_{\infty}(\mathcal{G}_n,m_n;P)\bigr),
\)
and write $\mathcal{C}_{P,n,(1)},\mathcal{C}_{P,n,(2)},\ldots$ for the
vertex sets of the connected components of
$\mathcal{G}_n[\mathcal{CS}_{P,n}]$, ordered by decreasing size. Ties are
resolved by a fixed measurable rule, and $\mathcal C_{P,n,(i)}:=\varnothing$
when there are fewer than $i$ components.
\begin{assumption}\label{assumption-MLWC}
Throughout, $((\cG_n,m_n))_{n\ge1}$ converges to
$(\cG,o,m)\sim\mu\in\mathcal P(\Gm)$ in marked local--weak convergence in
probability. We additionally assume that, for every fixed $R$, the local event
\begin{equation}\label{eq:continuity-set}
    E_R^P:=\{(G,o,m):o\in V_R(G,m;P)\} \text{ is a $\mu$-continuity set.}
\end{equation}
\end{assumption}
For finite mark spaces, \eqref{eq:continuity-set} always holds. For continuous type spaces, the condition amounts to requiring that the limiting rooted marked graph puts no mass on the boundary of the
local decision rule. 

\begin{remark}[Sufficient condition for Assumption \ref{assumption-MLWC}.] Let $A_1,\ldots,A_q$ be a finite Borel partition of $\mathscr M$,
and write $\ell(x)=j$ when $x\in A_j$. Suppose that for every graph
$H$, every vertex $v\in V(H)$, and every pair of markings $m,m'$ of $H$,
\[
  \ell(m(u))=\ell(m'(u))\ \text{for every }u\in B_r(H,v)
  \quad\Longrightarrow\quad
  P(v;H,m)=P(v;H,m').
\] 
Put $D=\bigcup_{j=1}^q\partial A_j$. For the marked local limit $\mu$, we claim that
\[
  \mu(m(o)\in D)=0
  \quad\Longrightarrow\quad
  \mu(\partial E_R^P)=0 \quad\text{for every } R\ge1.
\]
Assume that $\mu(m(o)\in D)=0$. Fix $R\ge1$ and put $s=rR$.
Since $\mu$ is a marked local weak limit of uniformly rooted finite graphs, it is unimodular. By the mass-transport principle,
\[
 \mathbb E_\mu\!\left[\sum_{v\in B_s(\mathcal G,o)}
                       \mathbf 1_{\{m(v)\in D\}}\right]
 =
 \mathbb E_\mu\!\left[|B_s(\mathcal G,o)|\,
                       \mathbf 1_{\{m(o)\in D\}}\right]
 =0,
\]
where the last equality follows from $\mu(m(o)\in D)=0$.
Thus, almost surely, every mark in $B_s(\mathcal G,o)$ lies in the interior of its cell. Since the ball is finite, there is an $\varepsilon>0$ such that moving each of its marks by less than $\varepsilon$ preserves all its cell labels. Every sufficiently close rooted marked graph therefore has an isomorphic rooted radius-$s$ ball with the same labels. By Lemma~\ref{lem:finite-round-locality} applied to the rule on cell labels, its root has the same survival status through round $R$. Hence $\mu(\partial E_R^P)=0$.

For a concrete example, take real-valued marks, a threshold $a\in\mathbb R$,
and an integer $k\ge1$. Retain a vertex if at least $k$ of its neighbors
in the current graph $H$ have marks greater than $a$. Thus
\[
 P(v;H,m)=\mathbf{1}\!\left\{
   \#\{u\sim_H v:m(u)>a\}\ge k
 \right\}.
\]
The cells $(-\infty,a]$ and $(a,\infty)$ have just one boundary point, $a$. The sufficient condition above therefore reduces to $\mu(m(o)=a)=0$. The marks at different vertices may be dependent. 

\emph{Counterexample.} Take a single edge, choose its root uniformly, and give its endpoints independent uniform marks on $[0,1]$. Retain a vertex if at least one of its neighbors in the current graph has a rational mark. The root's unique neighbor has a rational mark with probability zero, so $\mu(E_1^P)=0$. Since both the rationals and the irrationals are dense, changing only that neighbor's mark by an arbitrarily small amount can make the root either survive or fail in round~1. Thus every realization lies in $\partial E_1^P$, and $\mu(\partial E_1^P)=1$.
\end{remark}

\begin{theorem}[Universal upper bound]\label{thm:upper-bound}
 Under Assumption~\ref{assumption-MLWC}, for every $\varepsilon>0$,
\begin{equation}\label{eq:upper-bound-core}
\lim_{n\to \infty}\prob\left(\frac{|\mathcal{CS}_{P,n}|}{n}\ge\eta+\varepsilon\right) = 0.
\end{equation}
\end{theorem}
The upper bound is sharp precisely when finite-round survival for large $R$ approximates membership in the true surviving set. We formulate this through the following almost-locality conditions.

\begin{condition}
\label{cond:core}
$\lim_{R\to\infty}\limsup_{n\to\infty}
\prob\bigl(o_n\in V_R(\cG_n,m_n,P)\setminus\mathcal{CS}_{P,n}\bigr)=0.$
\end{condition}

\begin{condition}
\label{cond:giant}
$\lim_{R\to\infty}\limsup_{n\to\infty}
\prob\bigl(o_n\in V_R(\cG_n,m_n,P)\setminus\mathcal{C}_{P,n,(1)}\bigr)=0.
$
\end{condition}

\begin{theorem}[Almost-locality and size of the surviving core]
\label{thm:almost-local}
Suppose that Assumption~\ref{assumption-MLWC} holds.
If Condition~\ref{cond:core} holds, then
\begin{equation}\label{eq:core-lln}
\frac{|\mathcal{CS}_{P,n}|}{n}
\xrightarrow{\scriptscriptstyle\prob}\eta.
\end{equation}
Conversely, Condition~\ref{cond:core} is necessary for
\eqref{eq:core-lln}. If Condition~\ref{cond:giant} holds, then
\begin{equation}\label{eq:giant-lln}
\frac{|\mathcal{C}_{P,n,(1)}|}{n}
\xrightarrow{\scriptscriptstyle\prob}\eta,
\qquad
\frac{|\mathcal{C}_{P,n,(i)}|}{n}
\xrightarrow{\scriptscriptstyle\prob}0
\quad\text{for every }i\ge2.
\end{equation}
Conversely, Condition~\ref{cond:giant} is necessary for the first
convergence in~\eqref{eq:giant-lln}.
\end{theorem}
\begin{remark}[Infinite components of the surviving core]
\label{rem:infinite-surviving-component}
The survival probability $\eta$ in Theorem~\ref{thm:almost-local}
includes the contribution from finite components of the surviving
core. Let $\theta$ denote the $\mu$-probability that the root belongs
to an infinite component of $G_\infty(\mathcal G,m;P)$.

For $R\ge1$, let $L_R(G,m)$ be the set of vertices that, after $R$
deletion rounds on $(G,m)$, belong to a component of $G_R$ with at
least $R$ vertices. Write $L_{n,R}=L_R(\mathcal G_n,m_n)$ and
$C_i=\mathcal C_{P,n,(i)}$. Under Assumption~\ref{assumption-MLWC},
the condition
\begin{equation}\label{eq:infinite-component-locality}
  \lim_{R\to\infty}\limsup_{n\to\infty}
  \prob\bigl(o_n\in L_{n,R}\setminus C_1\bigr)=0
\end{equation}
implies that $|C_1|/n\xrightarrow{\scriptscriptstyle\prob}\theta$
and $|C_2|/n\xrightarrow{\scriptscriptstyle\prob}0$.
\shortversion{%
These limits follow by adapting the proof of
Theorem~\ref{thm:almost-local} to the finite-round sets $L_{n,R}$.}

Since $L_{n,R}\subseteq V_R(\mathcal G_n,m_n;P)$,
Condition~\ref{cond:giant} implies
\eqref{eq:infinite-component-locality}. The converse can fail.
For example, apply the $2$-core rule to the disjoint union of a
cycle of length $3j$ and $j$ triangles, with all vertex marks equal.
All $6j$ vertices survive, and exactly half lie outside $C_1$.
Thus the probability in Condition~\ref{cond:giant} is $1/2$ for
every $R$. For $R\ge4$, the triangles contribute no vertices to
$L_{n,R}$, so $L_{n,R}\setminus C_1=\varnothing$. As $j\to\infty$,
the rooted local limit is a bi-infinite path with probability $1/2$
and a triangle with probability $1/2$. Hence $\eta=1$ and
$\theta=1/2$.

\longversion{%
Fix $R\ge1$. The root belongs to $L_R$ precisely when it lies in a
connected set of $R$ round-$R$ survivors. Every vertex of such a set
lies in $B_{R-1}(\mathcal G,o)$. Thus the edges and the round-$R$
survival indicators in this ball determine whether $o\in L_R$.
By Lemma~\ref{lem:finite-round-locality}, the original marked ball
of radius $rR+R-1$ about $o$ determines all these indicators.

To check continuity, give each vertex $v$ the additional mark
$b_R(v)=\mathbf{1}_{\{(\mathcal G,v,m)\in\partial E_R^P\}}$.
Assumption~\ref{assumption-MLWC} gives $b_R(o)=0$ almost surely.
This marking is measurable and invariant under isomorphisms, so it
preserves unimodularity. By \cite[Lemma~2.3]{AldousLyons2007},
almost surely $b_R(v)=0$ for every vertex in the root component.
Since $B_{R-1}(\mathcal G,o)$ is finite, every sufficiently close
rooted marked graph has an isomorphic rooted $(R-1)$-ball on which
all the round-$R$ survival indicators agree. Hence $\{o\in L_R\}$
is a $\mu$-continuity set, and marked local--weak convergence gives
\[
  \frac{|L_{n,R}|}{n}\xrightarrow{\scriptscriptstyle\prob}p_R,
  \qquad p_R:=\mu\bigl(o\in L_R(\mathcal G,m)\bigr).
\]

The events $\{o\in L_R\}$ decrease with $R$, since the surviving
graphs shrink and the size threshold increases. A root that is
eventually deleted drops out, and a root in an infinite final
component belongs to every $L_R$. Suppose that the root survives
in a finite final component $K$. By local finiteness, $K$ has only
finitely many external neighbors in $\mathcal G$. Each is eventually
deleted, since any surviving neighbor would belong to $K$. Once
they have all been deleted, the root's component is exactly $K$.
It therefore fails the size test for all sufficiently large $R$.
This proves that $p_R\downarrow\theta$.

Finally, put $D_{n,R}=|L_{n,R}\setminus C_1|$. Every final component
of size at least $R$ is contained in $L_{n,R}$, so
\[
  |L_{n,R}|-D_{n,R}\le |C_1|\le |L_{n,R}|+R,
  \qquad
  |C_2|\le D_{n,R}+R.
\]
Since $o_n$ is uniform,
$\expec[D_{n,R}/n]=\prob(o_n\in L_{n,R}\setminus C_1)$.
By \eqref{eq:infinite-component-locality} and Markov's inequality,
\[
  \lim_{R\to\infty}\limsup_{n\to\infty}
  \prob(D_{n,R}>\varepsilon n)=0
  \qquad\text{for every }\varepsilon>0.
\]
The two asserted limits now follow from the bounds above and
$p_R\downarrow\theta$.
}
\end{remark}
Fix $s\ge1$ and let $Q$ be an arbitrary $s$-local property, evaluated in the original marked graph $(\mathcal G_n,m_n)$. For
$X\subseteq V(\mathcal G_n)$, let $\mathcal N_Q(X)$ be the set of vertices in
$X$ satisfying $Q$. Assume that, for every fixed $R$, the event
\begin{equation}\label{eq:continuity-set-property}
    E_R^{P,Q}:=\{(G,o,m):o\in V_R(G,m;P)\text{ and }o\text{ satisfies }Q\} \text{ is a $\mu$-continuity set.}
\end{equation}
Define
\(
\eta_R^Q:=\mu(E_R^{P,Q}),\
\eta^Q:=\lim_{R\to\infty}\eta_R^Q,
\)
where the limit exists because the events $E_R^{P,Q}$ decrease with $R$.

\begin{theorem}[Local statistics on the surviving core]
\label{thm:surviving-types}
Suppose Assumption~\ref{assumption-MLWC} and \eqref{eq:continuity-set-property} hold.
\begin{enumerate}
\item If \eqref{eq:core-lln} holds, then
\(
\frac{|\mathcal N_Q(\mathcal{CS}_{P,n})|}{n}
\xrightarrow{\scriptscriptstyle\prob}\eta^Q.
\)
\item If the first convergence in \eqref{eq:giant-lln} holds, then
\(
\frac{|\mathcal N_Q(\mathcal C_{P,n,(1)})|}{n}
\xrightarrow{\scriptscriptstyle\prob}\eta^Q.
\)
\end{enumerate}
\end{theorem}

\begin{remark}
    Theorem~\ref{thm:surviving-types} shows that once the density of the final $P$-core, or of its largest component, is known to converge to $\eta$, a wide array of convergence results follow. For example, related results are available for geometric graph models. Convergence of the proportion of vertices in the giant component has been established for random geometric graphs and soft random geometric graphs in~\cite{Penrose2003,Penrose2022}. Moreover, \cite{HHM23} established local--weak convergence for a broad class of geometric graph models, including sparse geometric inhomogeneous random graphs. Their proof can be adapted to retain information about spatial locations as marks.
    Consequently, for such models, the theorem yields convergence of proportion of vertices in that component whose spatial marks lie in a Borel set $A\subseteq\mathbb{R}^d$.
\end{remark}

\begin{remark}
Applications of this general framework are illustrated in the companion paper~\cite{ChatterjeeDharaSchulman2026}. In \cite{ChatterjeeDharaSchulman2026}, Condition~\ref{cond:core} is established for inhomogeneous random graph models and broad classes of cascading properties, consequently yielding concentration of the size of the surviving core.
\end{remark}

\section{Proofs}

\begin{proof}[Proof of \Cref{thm:upper-bound}]
For fixed $R$, set $A_R:= V_R(\cG_n,m_n,P),$ and $Z_R:=\frac{|A_R|}{n}$.
By \Cref{lem:finite-round-locality}, membership in $A_R$ is determined by a
finite rooted neighborhood. Assumption~\ref{assumption-MLWC} therefore gives
\[
Z_R\xrightarrow{\scriptscriptstyle\prob}\eta_R,
\]for each fixed $R$.
Moreover, $\mathcal{CS}_{P,n}\subseteq A_R$ for every $R$. Given
$\varepsilon>0$, choose $R$ so large that
\(
\eta_R\le\eta+\frac{\varepsilon}{2},
\)
which is possible since $\eta_R\downarrow\eta$, as $R\to\infty$. Then, as $n\to\infty$,
\[
\prob\left(\frac{|\mathcal{CS}_{P,n}|}{n}\ge\eta+\varepsilon\right)
\le \prob(Z_R\ge\eta+\varepsilon)\le \prob\left(Z_R\ge\eta_R+\frac{\varepsilon}{2}\right)
\to 0.
\]
\end{proof}

We next introduce a common notation for the two partitions used in the proof of \Cref{thm:almost-local}. For fixed $R$, write $A_R:= V_R(\cG_n,m_n,P)$. Consider either
\begin{enumerate}[label=(\roman*)]
    \item the \emph{core partition}, whose first block is $\mathcal{CS}_{P,n}$ and whose remaining blocks are the singleton vertices in $V(\cG_n)\setminus\mathcal{CS}_{P,n}$; or
    \item the \emph{component partition}, whose first blocks are the connected components of $\mathcal{CS}_{P,n}$, ordered by decreasing size, and whose remaining blocks are the singleton vertices in $V(\cG_n)\setminus\mathcal{CS}_{P,n}$.
\end{enumerate}
For the core partition, let $B_1:=\mathcal{CS}_{P,n}$. For the component
partition, let $B_1:=\mathcal C_{P,n,(1)}$; when the final core is empty, we
use the convention $B_1=\varnothing$ and regard all vertices as singleton
blocks. With this convention, write $\mathcal B=(B_i)_{i\ge1}$ for the
resulting block family and define
\[
D_{\mathcal{B}}(R):=|A_R\setminus B_1|,
\qquad
M_{\mathcal{B}}(R):=\#\{(u,v)\in A_R \times A_R:u,v\text{ lie in different nonempty blocks of }\mathcal{B}\}.
\]

\begin{lemma}[Partition estimate]\label{lem:partition-estimate}
For both partitions,
\[
|A_R|D_{\mathcal{B}}(R)-n
\le M_{\mathcal{B}}(R)
\le 2|A_R|D_{\mathcal{B}}(R).
\]
\end{lemma}

\begin{proof}
Write
\(
a:=|A_R|,\ b:=|B_1|,\ D:=D_{\mathcal B}(R).
\)
Since $B_1\subseteq\mathcal{CS}_{P,n}\subseteq A_R$, we have $a=b+D$.

For the core partition, every vertex of $A_R\setminus B_1$ is a singleton
block. Since $M_{\mathcal B}(R)$ counts ordered pairs in different blocks,
\[
M_{\mathcal B}(R)=2bD+D(D-1)=(a+b-1)D.
\]
This is at most $2aD$. If $b\ge1$, it is at least $aD$; if $b=0$, then
$D=a$ and
\[
M_{\mathcal B}(R)=a(a-1)=aD-a\ge aD-n.
\]

For the component partition, first suppose the surviving core is nonempty, i.e.,
$b\ge1$. For $i\ge2$, set
\(
s_i:=|A_R\cap B_i|.
\)
For surviving-component blocks this equals $|B_i|$, while for singleton blocks
outside the final core it can be $0$ or $1$. Since the blocks partition
$A_R\setminus B_1$,
\(
\sum_{i\ge2}s_i=D.
\)
The exact ordered-pair count is
\[
M_{\mathcal B}(R)
=a^2-b^2-\sum_{i\ge2}s_i^2
=2bD+D^2-\sum_{i\ge2}s_i^2.
\]
Dropping the last term yields
\[
M_{\mathcal B}(R)\le2bD+D^2=(a+b)D\le2aD.
\]
Since $B_1$ is a largest surviving component and all other nonsurviving blocks
are singletons, $s_i\le b$ for every $i\ge2$. Therefore
\[
\sum_{i\ge2}s_i^2\le b\sum_{i\ge2}s_i=bD,
\]
and hence
\[
M_{\mathcal B}(R)\ge2bD+D^2-bD=(b+D)D=aD.
\]
If the surviving core is empty, all blocks are singletons, $D=a$, and
\[
M_{\mathcal B}(R)=a(a-1)\ge aD-n.
\]
Combining the cases proves the result.
\end{proof}

\begin{lemma}[Equivalent form of 
almost-locality]\label{lem:equivalence}
Assume that $\eta>0$. For either partition,
\[
\lim_{R\to\infty}\limsup_{n\to\infty}\frac1{n^2}\expec[M_{\mathcal{B}}(R)]=0
\quad\Longleftrightarrow\quad
\lim_{R\to\infty}\limsup_{n\to\infty}\frac1n\expec[D_{\mathcal{B}}(R)]=0.
\]
For the core partition and component partition, the second condition is equivalent to Condition~\ref{cond:core} and Condition~\ref{cond:giant}, respectively.
\end{lemma}

\begin{proof}
For fixed $R$, let $X_n:=|A_R|/n$ and $Y_n:=D_{\mathcal B}(R)/n$. Since
$X_n\xrightarrow{\scriptscriptstyle\prob}\eta_R$, $0\le X_n,Y_n\le1$, and
bounded convergence in probability implies $L^1$ convergence,
\[
\left|\expec[X_nY_n]-\eta_R\expec[Y_n]\right|
\le\expec|X_n-\eta_R|\longrightarrow0.
\] 
Using the upper and lower bounds in Lemma~\ref{lem:partition-estimate}, dividing by $n^2$ and taking $\limsup_{n\to\infty}$
therefore shows that, for every fixed $R$,
\[
\eta_R\limsup_{n\to\infty}\expec[Y_n]
\le
\limsup_{n\to\infty}\expec\left[\frac{M_{\mathcal B}(R)}{n^2}\right]
\le
2\eta_R\limsup_{n\to\infty}\expec[Y_n].
\]
Since $\eta_R\downarrow\eta>0$, the two
iterated limits vanish simultaneously. Finally,
\[
\expec\left[\frac{D_{\mathcal B}(R)}n\right]
=\prob(o_n\in A_R\setminus B_1),
\]
which gives the stated equivalence with Conditions~\ref{cond:core} and
\ref{cond:giant}.
\end{proof}
For random variables $(\Delta_{R,n})_{R,n\ge1}$, write
$\Delta_{R,n}=o_{R,n,\mathbb P}(1)$ if, for every $\varepsilon>0$,
\[
\lim_{R\to\infty}\limsup_{n\to\infty}
\prob(|\Delta_{R,n}|>\varepsilon)=0.
\]

\begin{lemma}[Sum-of-squares estimate]\label{lem:sum-squares}
Assume $\eta>0$. Under Condition~\ref{cond:core} for the core partition, and under Condition~\ref{cond:giant} for the component partition,
\[
\frac1{n^2}\sum_{i\ge1}|B_i|^2\1_{\{B_i\subseteq A_R\}}
=
\eta^2+o_{R,n,{\sss \PR}}(1).
\]
\end{lemma}

\begin{proof}
For both partitions, membership in $A_R$ is constant on each block.
Indeed, every block corresponding to a component of the final surviving
core is contained in $A_R$, since $\mathcal{CS}_{P,n}\subseteq A_R$.
All remaining blocks are singletons, and such a singleton may or may not
belong to $A_R$.
Hence
\[
Z_R:=\frac{|A_R|}{n}
=\frac1n\sum_{i\ge1}|B_i|\1_{\{B_i\subseteq A_R\}}.
\]
Squaring and separating ordered pairs according to whether they lie in the same
block gives
\[
Z_R^2
=\frac1{n^2}\sum_{i\ge1}|B_i|^2\1_{\{B_i\subseteq A_R\}}
+\frac{M_{\mathcal B}(R)}{n^2}.
\]
By \Cref{lem:equivalence} and Markov's inequality,
\(
\frac{M_{\mathcal B}(R)}{n^2}=o_{R,n,\mathbb P}(1).
\)
Also $Z_R\xrightarrow{\scriptscriptstyle\prob}\eta_R$ for fixed $R$ and
$\eta_R\downarrow\eta$. The conclusion follows.
\end{proof}

\begin{proof}[Proof of \Cref{thm:almost-local}]
For the core partition, set
\(
D_R^{\rm core}:=|A_R\setminus\mathcal{CS}_{P,n}|.
\)
Since $\mathcal{CS}_{P,n}\subseteq A_R$,
\[
\frac{|\mathcal{CS}_{P,n}|}{n}
=\frac{|A_R|}{n}-\frac{D_R^{\rm core}}n.
\]
Condition~\ref{cond:core} and Markov's inequality imply
$D_R^{\rm core}/n=o_{R,n,\mathbb P}(1)$. Combining this with
$|A_R|/n\xrightarrow{\scriptscriptstyle\prob}\eta_R$ for fixed $R$ and
$\eta_R\downarrow\eta$ proves \eqref{eq:core-lln}.

The same argument with
\(
D_R^{\rm giant}:=|A_R\setminus\mathcal C_{P,n,(1)}|
\)
proves the first convergence in \eqref{eq:giant-lln} under
Condition~\ref{cond:giant}. Since this condition implies
Condition~\ref{cond:core}, we also have
$|\mathcal{CS}_{P,n}|/n\xrightarrow{\scriptscriptstyle\prob}\eta$. Therefore
\[
\frac{|\mathcal C_{P,n,(2)}|}{n}
\le
\frac{|\mathcal{CS}_{P,n}|-|\mathcal C_{P,n,(1)}|}{n}
\xrightarrow{\scriptscriptstyle\prob}0,
\]
and the same holds for each fixed $i\ge2$.

For necessity in the core case, assume
$|\mathcal{CS}_{P,n}|/n\xrightarrow{\scriptscriptstyle\prob}\eta$. Since all
normalized variables are bounded by $1$, convergence in probability implies
convergence of expectations. For fixed $R$,
\[
\begin{aligned}
\prob(o_n\in A_R\setminus\mathcal{CS}_{P,n})
&=\expec\left[\frac{|A_R|-|\mathcal{CS}_{P,n}|}{n}\right]\to \eta_R-\eta.
\end{aligned}
\]
Letting $R\to\infty$ proves Condition~\ref{cond:core}, including when
$\eta=0$. Necessity for largest-components uses identical argument by replacing
$\mathcal{CS}_{P,n}$ by $\mathcal C_{P,n,(1)}$.
\end{proof}

\begin{proof}[Proof of \Cref{thm:surviving-types}]
We prove the first assertion; the second is identical. For fixed $R$, define
\[
Z_R^Q:=\frac1n\sum_{v\in V(\cG_n)}
\1_{\{v\in A_R,\ v\text{ satisfies }Q\}}.
\]
By the continuity assumption on $E_R^{P,Q}$ and marked local--weak convergence,
\(Z_R^Q\xrightarrow{\scriptscriptstyle\prob}\eta_R^Q.
\)
Moreover,
\[
0\le Z_R^Q-
\frac{|\mathcal N_Q(\mathcal{CS}_{P,n})|}{n}
\le
\frac{|A_R\setminus\mathcal{CS}_{P,n}|}{n}.
\]
By the necessity part of \Cref{thm:almost-local}, \eqref{eq:core-lln} implies
Condition~\ref{cond:core}; hence the right-hand side is
$o_{R,n,\mathbb P}(1)$. Letting first $n\to\infty$ and then $R\to\infty$
gives
\[
\frac{|\mathcal N_Q(\mathcal{CS}_{P,n})|}{n}
\xrightarrow{\scriptscriptstyle\prob}\eta^Q.
\]
For the largest component, use the necessity of Condition~\ref{cond:giant} and
replace $\mathcal{CS}_{P,n}$ by $\mathcal C_{P,n,(1)}$.
\end{proof}

\begingroup
\small
\setstretch{.8}
\bibliographystyle{abbrv}
\bibliography{refs}
\endgroup

\end{document}